\documentclass [12pt]{article}
\usepackage[utf8]{inputenc}
\usepackage{amsmath}
\usepackage{amsfonts, amssymb}

\newenvironment{proof}[1][] {\noindent {\bf Proof#1:} }{\hspace*{\fill}$\square$\medskip\par}

\newtheorem{thm}{Theorem}
\newtheorem{lem}[thm]{Lemma}
\newtheorem{prop}[thm]{Proposition}

\newtheorem{rem}[thm]{Remark}

\def\R{{\mathbb R}}
\def\E{{\mathbb E\,}}

\def\P{{\mathbb P}}

\def\be#1\ee{\begin{equation}#1\end{equation}}

\def\BB{\mathcal B}
\def\crr{{c_{\scriptscriptstyle R}}}
\def\cuu{{c_{\scriptscriptstyle U}}}
\def\ed#1{ {\mathbf 1}_{ \{#1  \}}}             
\def\eps{\varepsilon}
\def\EUU{{\mathbb E}U^2\,}
\def\EU{{\mathbb E}\,U\,}
\def\ERR{{\mathbb E}R^2\,}
\def\ER{{\mathbb E}R\,}
\def\mul{\mu^{(\ell)}}
\def\mulr{\mu^{(\ell,r)}}
\def\RR{\mathcal R}
\def\ss{{\mathbf s}}
\def\SS{\mathcal S}
\def\tN{{\widetilde  N}}
\def\tpf{{\bar F_R}}
\def\Var{\textrm{Var}\,}
\def\vro{\varrho}

\title{Large Deviations of Telecom Processes}
\author{M.A. Lifshits \and S.E. Nikitin}
\date{July 2021}

\begin{document}

\maketitle

\begin{abstract}
We study large deviation properties of
Telecom processes appearing as limits in a critical regime
of infinite source Poisson models.
\end{abstract}

{\bf AMS Subject Classification:}\
Primary: 60F10. 
Secondary: 60F05, 
60G57, 
60E07. 

{\bf Keywords:}\ large deviation probabilities, Telecom process,
Poisson random measure, teletraffic, workload.

\section{Introduction: Telecom processes}

\subsection{A service system}

Telecom processes originate from a remarkable work by I. Kaj and M.S. Taqqu
 \cite{KT} who handled the limit behavior of "teletraffic systems" by using
the language of integral representations as a unifying technique. Their
article brightly represents a wave of interest to the subject, see
e.g. \cite{KLNS,Kur,PipT0,RH,Taq02},
and the surveys with further references \cite{K02,K05,K06},
to mention just a few.
Simplicity of the dependence mechanism used in the model enables to get a
clear understanding both of long range dependence in one case, and
independent increments, in other cases.

The work of the system represents a collection of \emph{service processes} or
\emph{sessions},
using telecommunication terminology.
Every process starts at some time $s$, lasts $u$ units of time, and
occupies $r$ \emph{resource} units (synonyms for resource  are
\emph{reward, transmission rate} etc). The amount of occupied resources $r$
remains constant during every service process.



The formal model of the service system is based on Poisson random
measures and looks as follows. Let
$\RR:=\{(s,u,r)\}= \R \times \R_+\times \R_+$. Every point $(s,u,r)$
corresponds to a possible service process with starting time $s$, duration
$u$, and required resources $r$.

The system is characterized by the following parameters:
\begin{itemize}
\item $\lambda>0$ -- \emph{arrival intensity} of service processes;
\item $F_U(du)$ -- the distribution of service duration;
\item $F_R(dr)$ -- the distribution of amount of required resources.
\end{itemize}
One may assume $\P(R>0)=P(U>0)=1$ without loss of generality.

Define on $\RR$ an intensity measure
\[
     \mu(ds,du,dr)= \lambda ds\, F_U(du)\, F_R(dr).
\]
Let  $N$ be a Poisson random measure
with intensity $\mu$.
One may consider the samples of $N$ (sets of triplets $(s,u,r)$, each triplet corresponding to a service process) as variants (sample paths) of the work for the system.


The {\it instant workload} on the system at time $t$ writes as
\[  
     W^\circ(t)= \int_\RR  r \ed{s\le t\le s+u} d N.
\]
This is essentially  the sum of occupied resources over the processes active at time $t$.
The {\it integral workload} over the interval $[0,t]$ is
\begin{eqnarray*}  
   W^*(t)&=& \int_0^t W^\circ(\tau) d\tau
    =  \int_\RR  r \int_0^t \ed{s\le \tau \le s+u} d\tau d N
\\  
    &=&  \int_\RR  r\cdot \Big| [s,s+u]\cap[0,t] \Big|  d N
    :=   \int_\RR  r \ell_t(s,u)  d N.
\end{eqnarray*}
Here $|\cdot|$ stands for the length of an interval, and the kernel
\begin{equation} \label{ell}
     \ell_t(s,u) := \Big| [s,s+u]\cap[0,t] \Big|
\end{equation}
will be often used in the sequel.

Notice that $W^\circ(\cdot)$ is a stationary process and its integral
$ W^*(\cdot)$ is a process with stationary increments.

We suppose that either the variables $R$ and $U$ have finite variance,
or their distributions have regular tails. More precisely, either
\[
   \P(U>u)\sim \frac{\cuu}{ u^{\gamma}}\ ,
   \qquad u\to \infty, \qquad 1<\gamma<2, \ \cuu>0,
\]
or $\EUU< \infty$. In the latter case we formally set $\gamma:=2$.

Analogously, we assume either
\[
   \P(R>r)\sim \frac{\crr}{ r^{\delta}}\ ,
   \qquad r\to \infty, \qquad 1<\delta<2,\ \crr>0,
\]
or $\ERR< \infty$. In the latter case we formally set $\delta:=2$.

The behavior of the service system crucially
depends of the parameters $\gamma, \delta\in (1,2]$.

\subsection{Limit theorems for the workload}
\label{ss:ltw}

\subsubsection{Centered and scaled workload process}
\label{sss:Z}


The main object of theoretical interest is the behavior
of the integral workload as a process (function of time) observed on
long time intervals. 

In order to obtain a meaningful limit, one must 

scale (contract) the time so that it would run through the
standard time interval,
center the workload process,
and divide it by an appropriate scalar factor.

We choose $[0,1]$ as a standard time interval. 

Centering and scaling by appropriate factor $b$ lead to a
\emph{normalized  integral workload process}
\[
     Z_a(t):=\frac{W^*(at)-\ER\cdot \EU\cdot a \lambda  t}{b}\ ,
     \qquad t\in[0,1], \ b=b(a,\lambda).
\]

\subsubsection{A limit theorem leading to a Telecom process}
\label{sss:Y}


It is remarkable that a simple tuning of three parameters  $\lambda, \gamma, \delta$
may lead to different limiting processes for $Z_a$, namely, one can obtain
\begin{itemize}
    \item a Wiener process;
    \item a fractional Brownian motion with index $H\in(1/2,1)$;
    \item a centered L\'evy stable process with positive spectrum;
    \item a stable Telecom process;
    \item a Poisson Telecom process.
\end{itemize}
While the first three processes present a core of the classical theory of stochastic processes, the Telecom processes are almost not studied. In this article we focus on some key properties of the Poisson Telecom process.

For the full panorama of related limit theorems we refer to \cite[Chapter 3]{LifEx} and recall here only one result concerning the Poisson Telecom process (cf. \cite[Theorem 13.16]{LifEx}) related to the case of \emph{critical intensity}
\be \label{critint}
    \frac{\lambda}{a^{\gamma-1}}\to L,  \qquad\qquad  0<L<\infty.
\ee

\begin{thm} \label{t:ltpT}
Assume that $1<\gamma<\delta\le 2$,  $a\to\infty$,
and that critical intensity condition $\eqref{critint}$ holds.
Let $Q:=L \, \cuu\, \gamma$.
Then with scaling $b:=a$ the finite dimensional distributions of the
process $Z_a$  converge to those of the Poisson Telecom process $Y_{Q,\gamma}$
admitting an integral representation
\[
   Y_{Q,\gamma}(t) =\int_\RR r\, \ell_t(s,u) \bar N_{Q,\gamma}(ds,du,dr).
\]
Here $\ell_{t}(s,u)$ is the kernel defined in $(\ref{ell})$ and
$\bar N_{Q,\gamma}$ is a centered Poisson random measure of intensity 
$Q\,\mu_{\gamma}$ where 
\[
   \mu_{\gamma}(ds,du,dr) := \frac{ds \, du} {u^{\gamma+1}}\ F_R(dr).
\]
\end{thm}

For studies on Poisson Telecom process we refer to \cite{CT,Gai}.
It is well known that the process $(Y_{Q,\gamma}(t))_{t\ge 0}$, is correctly defined if 
$\E (R^\gamma)<\infty$. In accordance with its role in the limit theorem, it has stationary increments. It is, however, not self-similar like other limiting processes in the same model, such as Wiener process, fractional Brownian motion, or strictly stable L\'evy process.

\section{Main results} \label{s:results}

\subsection{A limit theorem for Telecom process } \label{ss:lt}
At large time scales the Poisson Telecom process essentially behaves as a $\gamma$-stable L\'evy
process. This fact is basically known but we present it here for completeness of exposition. The analogy with a stable law will also guide us (to some extent and within a certain range) in the subsequent studies of large deviation probabilities.

\begin{prop} \label{p:lt} We have a weak convergence
\be \label{eq:lt}
 \left( \E( R^\gamma)\, t \right)^{-1/\gamma} \,   Y_{Q,\gamma}(t) \ \Rightarrow \SS_{Q,\gamma}, \qquad \textrm{ as } t\to \infty,
\ee
where $\SS_{Q,\gamma}$ is a centered strictly $\gamma$-stable random variable with positive spectrum,
i.e.
\[
   \E \exp\{ it \SS_{Q,\gamma}\}  =\exp\left\{ Q \int_0^\infty \frac{e^{i t u}-1-i t u }{ u^{\gamma+1}} \, du \right\}
\]
\end{prop}

\subsection{Large deviations} \label{ss:ld}
According to the limit theorem \eqref{p:lt}, large deviation probability is
\[
   \P( Y_{Q,\gamma}(t)\ge \vro),  \quad \textrm{as }  \ \vro=\vro(t)\gg t^{1/\gamma}.
\]
Its behavior  may be different in different zones of $\vro$ and may depend
on the distribution of $R$. We consider the most important cases in the following subsections.

\subsubsection{Moderate large deviations} \label{sss:moderate}

\begin{thm} \label{t:moderate}
Let $\vro=\vro(t)$ be such that $t^{1/\gamma}\ll \vro \ll t$.
Then 
\be \label{eq:moderate}
    \P( Y_{Q,\gamma}(t)\ge \vro) = D \, t\, \vro^{-\gamma} \ (1+o(1)),  
    \qquad \textrm{as } \ t\to\infty,
\ee
where $D:=\tfrac{Q \ \E(R^\gamma)}{\gamma}$.
\end{thm}

This result should be compared with the limit theorem \eqref{eq:lt} because \eqref{eq:moderate} yields
\begin{eqnarray*}
 &&  \P \left( (\E(R^\gamma)t)^{-1/\gamma} Y_{Q,\gamma}(t)\ge \rho\right) 
   = \P\left( Y_{Q,\gamma}(t)\ge (\E(R^\gamma)t)^{1/\gamma} \rho\right) 
\\
   &\sim& D\, t\, (\E(R^\gamma)t)^{-1} \rho^{-\gamma}
   = \frac {Q}{\gamma}\ \rho^{-\gamma} \sim \P(\SS_{Q,\gamma}\ge \rho),
\end{eqnarray*}
whenever $1\ll \rho\ll t^{- (\gamma-1)/\gamma}$. In other words, the moderate large deviation
probabilities are equivalent to those of the limiting distribution.

Using the terminology of the background service system, moderate deviation is attained by a unique
heavy service process. We will stress this fact later in the proof.

\subsubsection{Intermediate large deviations} \label{sss:intermed}

The following result describes the situation on the upper boundary of
moderate deviations' zone. 

\begin{thm} \label{t:intermed}
Let $\kappa>0$ be such that $\P(R\ge \kappa)>0$ and
\be \label{pkappa0}
    \P(R=\kappa)=0.
\ee
Let $\vro=\vro(t)=\kappa t$. Then 
\[
    \P( Y_{Q,\gamma}(t)\ge \vro) = Q D^{(1)}_I(\kappa) \, t^{-(\gamma-1)} \, (1+o(1)),  
    \qquad \textrm{as } \ t\to\infty,
\]
where 
\[
   D_I^{(1)}(\kappa) := \left( \frac{\kappa^{-\gamma}}{\gamma} \, \E(R^\gamma\ed{R\ge \kappa})
   + \frac{(2-\gamma)\kappa^{1-\gamma}}{(\gamma-1)\gamma}\,   \E(R^{\gamma-1} \ed{R\ge \kappa})
   \right).
\]
\end{thm}

\begin{rem} {\rm
There is a continuity between the moderate and intermediate zones in what concerns the degree of $t$ but the constant in the intermediate case is different. Indeed, by plugging formally $\rho:=\kappa t$ 
into \eqref{eq:moderate} one obtains the asymptotics 
$\tfrac{Q\, \E(R^\gamma)}{\gamma}\, \kappa^{-\gamma}\, t^{-(\gamma-1)}$ which corresponds to the fist term in
the definition of $D_I^{(1)}(\kappa)$. When $\kappa$ goes to zero, the second term in that definition 
is smaller than the first one because
\[
   \kappa\, \E(R^{\gamma-1} \ed{R\ge \kappa}) \le  \E(R^{\gamma} \ed{R\ge \kappa})
   \searrow 0, \qquad \textrm{as } \kappa\to 0.
\]
}\end{rem}

\begin{rem} {\rm
If \eqref{pkappa0} does not hold, the decay order of large deviations
will be the same but the expression for the corresponding constant becomes more involved and less explicit. 
}\end{rem}

\medskip

The attentive reader will notice that Theorem \ref{t:intermed} does not work for large $\kappa$ if the distribution of $R$ is compactly supported. Indeed, in this case the large deviation asymptotics will be different,
as the next result shows. In terms of the service system, it handles the case when the large deviation can be attained by accumulation of 
$n$ heavy service processes but cannot be attained by $(n-1)$ ones.

\begin{thm} \label{t:intermed2}
Let $\kappa>0$. Let $n$ be the positive integer such that $\P(R\ge \tfrac{\kappa}{n})>0$ but
\be \label{choosezeta}
  \P(R\ge \tfrac{\kappa}{n-\zeta})=0
  \qquad \textrm{for some } \zeta\in(0,1). 
\ee 
Assume that
\be \label{pkappa0n}
  \P(R_1+\cdots+R_n=\kappa)=0,
\ee
where $R_1,\dots, R_n$ are independent copies of $R$.

Let $\vro=\vro(t):=\kappa t$.
Then 
\[
    \P( Y_{Q,\gamma}(t)\ge \vro) =  Q^n D^{(n)}_I(\kappa) \, t^{-(\gamma-1)n} \, (1+o(1)),  
    \qquad \textrm{as } \ t\to\infty,
\]
where $D^{(n)}_I(\kappa)$ is some finite positive constant depending on $n,\kappa$ and on the law of $R$.
\end{thm}

\begin{rem} {\rm 
The explicit form of  $D^{(n)}_I(\kappa)$ is given in equation \eqref{DIn} below.
}
\end{rem}

\begin{rem} {\rm
Theorem \ref{t:intermed2} does not cover a critical case
$\zeta=1$, where we have $\P(R\ge \tfrac{\kappa}{n-1})=0$ but  $\P(R\ge \tfrac{\kappa}{n-1}-\eps)>0$ for all $\eps>0$. In this case, the
assertion of the theorem may not hold because the large deviation probability behavior depends of that of the upper tail
$\P(R\in [\tfrac{\kappa}{n-1}-\eps,\tfrac{\kappa}{n-1}))$, as $\eps\to 0$.
}
\end{rem}

\subsubsection{Ultralarge deviations}  \label{sss:ultraLD}

\begin{thm} \label{t:ultra_regular}
Let $\vro=\vro(t)\gg t$. Assume that the tail probability function $\tpf(y):=\P(R\ge y)$ is regularly varying of
negative order $-m$ where $m>\gamma$.
Then 
\[
    \P( Y_{Q,\gamma}(t)\ge \vro) 
    =  Q\, D \, t^{-(\gamma-1)} \tpf(\vro/t) \,  (1+o(1)),  
    \qquad \textrm{as } \ t\to\infty,
\]
where 
\[
   D:= \frac{m(m-1)}{\gamma(\gamma-1)(m-\gamma+1)(m-\gamma)}.
\]            
\end{thm}

As in Theorem \ref{t:moderate}, the workload's large deviation is attained by a unique long and heavy service process.

Theorem \ref{t:ultra_regular} deals with the distributions of $R$ having essentially polynomial tails. The corresponding distributions with light tails lead to completely different results such as Poisson large deviations. This direction requires supplementary research to be presented elsewhere. 

\subsection{Concluding remark}

A challenging case when the workloads' ultralarge  deviation is formed via the interaction of infinitely many service processes remains beyond the scope of this article. Here, a large deviation rate function related to the distribution of $R$ must play a major role and the results in the spirit of classical large deviation theory \cite{DZ} are expected. This might be a subject of a subsequent work.

\section{Proofs} \label{s:proofs}

\subsection{Preliminaries}

Let us introduce two auxiliary intensity measures.
The first one is the "distribution" of the kernel  $\ell_t$, namely
\[
   \mul_t (A):= \int_{\R} \int_{\R_+} \ed{\ell_t(s,u)\in A} \frac{du}{u^{\gamma+1}}\, ds,
   \qquad A\in\BB([0,t]).
\]
The second is the "distribution" of the product $r\ell_t(s,u)$,
\begin{eqnarray*}
\mulr_t (A)&:=&  \int_{\R_+} \int_{\R_+} \ed{ r \ell \in A} \mul_t(d\ell) F_R(dr)
\\
 &=& \mu\{(s,u,r): r \ell_t(s,u)\in A\},   
\qquad A\in\BB(\R_+).
\end{eqnarray*}
A simple variable change in the definition of $Y_{Q,\gamma}(t)$ yields
\be \label{TeleProc2}
    Y_{Q,\gamma}(t) = \int_{\R_+} v \tN_{Q,\gamma}(dv) 
\ee
where $\tN_{Q,\gamma}$ is a centered Poisson measure with intensity $Q\mulr_t$. Therefore,
the properties  of $\mulr_t$ determine those of $Y_{Q,\gamma}(t)$.

As a first step, we give an explicit formula for the intermediate measure  $\mul_t$.
First, by definition we have $ \mul_t(t,\infty) =0$. Next, let us fix an $\ell_0\in (0,t]$ and find $\mul_t[\ell_0,t]$. In fact,
$\ell_t(s,u)\ge \ell_0$ iff $u\ge \ell_0$ and 
$s\in[\ell_0-u,t-\ell_0]$. Therefore,
\be \label{mul}
   \mul_t[\ell_0,t] =\int_{\ell_0}^\infty (t-2\ell_0+u) \frac{du}{u^{\gamma+1}}
   = \frac{t\,\ell_0^{-\gamma}}{\gamma} +\frac{2-\gamma}{(\gamma-1)\gamma} \, \ell_0^{1-\gamma}.
\ee
It follows that the measure $\mul_t$ has a weight $\frac{t^{-(\gamma-1)}}{(\gamma-1)\gamma}$
at the right boundary point $t$ and a density
\[
 \frac{d\mul_t}{d\ell} (\ell) = t \, \ell^{-1-\gamma} +\frac{2-\gamma}{\gamma}\, \ell^{-\gamma},
\qquad 0<\ell<t.
\]
 For each $\ell_0>0$, formula \eqref{mul} also yields a bound
\be \label{mul_b}
   \mul_t[\ell_0,\infty) =   \mul_t[\ell_0,t] \le
    \frac{t\ell_0^{-\gamma}}{\gamma} \left( 1 +\frac{2-\gamma}{\gamma-1}\right)
    =  \frac{t\,\ell_0^{-\gamma}}{\gamma(\gamma-1)}.
\ee

Finally, consider the asymptotic behavior of
\be  \label{mulr_intrepr}
   \mulr_t[\vro,\infty)  = \int_{\R_+} \mul_t\left[ \frac{\vro}{r} ,t\right] F_R(dr).
\ee
Assume that $\vro\to\infty$ but $\vro/t\to 0$.
Then it follows from \eqref{mul} that for every fixed $r$
\be \label{conv_r}
   \mul_t\left[ \frac{\vro}{r} ,t\right] 
   = \frac{t\, \vro^{-\gamma} r^\gamma}{\gamma} (1+o(1)).
\ee
By using \eqref{mul_b}, we also have an integrable majorant w.r.t. the law $F_R$:
\[
    \mul_t\left[ \frac{\vro}{r} ,t\right]  
    \le  \frac{ t\, \vro^{-\gamma} r^\gamma} {\gamma(\gamma-1)}.
\]
By integrating this estimate in \eqref{mulr_intrepr} we obtain
\be \label{mulr_b}
   \mulr_t[\vro,\infty) \le \frac{\E(R^\gamma)} {\gamma(\gamma-1)} \, t\, \vro^{-\gamma}.
\ee
Furthermore, by Lebesgue's majorated convergence theorem \eqref{mulr_intrepr} and \eqref{conv_r} yield
\be \label{mulr_asymp}
  \mulr_t [\vro,\infty) 
   =  \vro^{-\gamma} \,  t \int_{\R_{+}} \frac{r^\gamma}{\gamma} F_R(dr)\,  (1+o(1)) 
   =   \frac{\E(R^\gamma)}{\gamma}\, t\, \vro^{-\gamma} \, (1+o(1)).
\ee

\subsection{Proof of Proposition \ref{p:lt} \label{ss:proof_prop_lt}}

Consider the integral representation \eqref{TeleProc2}. According to a general criterion of the weak
convergence of Poisson integrals to a stable law \cite[Corollary 8.5]{LifEx},
it is enough to check that for each fixed $\rho>0$
\be \label{conv}
  Q\ \mulr_t \{ v: ( \E(R^\gamma) \, t)^{-1/\gamma} v \ge\rho\} 
  = Q\ \frac{\rho^{-\gamma}}{\gamma} \  (1+o(1)) 
\ee
combined with the uniform bound
\be \label{conv_u}
      \sup_{t>0} \sup_{\rho>0} \rho^\gamma \ \mulr_t \{v: (\E(R^\gamma)\, t)^{-1/\gamma} v \ge\rho\} <\infty.  
\ee
Indeed, by substituting $\vro=\rho ( \E(R^\gamma) t)^{1/\gamma}$ in \eqref{mulr_asymp} we obtain \eqref{conv}
and by making the same substitution in \eqref{mulr_b} we obtain \eqref{conv_u}.
\hfill $\Box$

\subsection{A decomposition} 
Take some $v_0>0$ and split the integral representation \eqref{TeleProc2} into three parts:
\begin{eqnarray} \nonumber
    Y_{Q,\gamma}(t) &=& \int_{0}^{v_0} v \tN_{Q,\gamma}(dv)
            + \int_{v_0}^\infty v N_{Q,\gamma}(dv) 
            - Q \int_{v_0}^\infty v \mulr_t(dv)
\\  \label{split}
   &:=& Y^\circ(t) + Y^{\dag}(t)-E_t,
\end{eqnarray} 
where $N$ is the corresponding {\it non-centered} Poisson random measure and $E_t$ is the centering
deterministic function.

The variance of $ Y^\circ(t)$ admits an upper bound
\begin{eqnarray*}
   \Var  Y^\circ(t)  &=& Q  \int_{0}^{v_0} v^2 \mulr_t(dv) 
   = 2 Q  \int_{0}^{v_0} v\ \mulr_t[v,v_0]\, dv
\\
   &\le& 2 Q  \int_{0}^{v_0} v \ \mulr_t[v,\infty) \, dv.
\end{eqnarray*}
Using \eqref{mulr_b} we get
\be \label{varb}
   \Var  Y^\circ(t) \le \frac{2 \, Q \, t}{\gamma(\gamma-1)} \E(R^\gamma) \int_0^{v_0} v^{1-\gamma}\, dv
   = D_2 \, t \, v_0^{2-\gamma},
\ee
where $D_2:=  \tfrac{2Q}{\gamma(\gamma-1)(2-\gamma)}\, \E(R^\gamma)$.
\medskip

Similarly, the centering term admits a bound
\be \label{Ebound}
   0 \le E_t   \le Q \int_{v_0}^\infty  \mulr_t[v,\infty) \, dv 
   + Q \, v_0 \, \mulr_t[v_0, \infty) 
  \le D_1 \, t\, v_0^{1-\gamma},
\ee
where $D_1:=  \tfrac{Q}{(\gamma-1)^2} \, \E(R^\gamma)$. 

\subsection{A lower bound for large deviations}
\label{ss:LowerLD}
We will give a lower bound for large deviation probabilities $\P( Y_{Q,\gamma}(t)\ge \vro)$ 
with $\vro=\vro(t)\gg t^{1/\gamma}$.
Let $h,\delta$ be small positive numbers. Define $v_0:= h\vro$ and consider the corresponding
decomposition \eqref{split}.

First of all, notice that $E_t$ is negligible at the range $\vro$ because by \eqref{Ebound}
we have 
\[
   E_t\le D_1 t (h\vro)^{1-\gamma} =  D_1 h^{1-\gamma}   \left(t^{-1/\gamma}\vro\right)^{-\gamma} \vro 
   = o(\vro).
\]
Therefore, we may and do assume $t$ to be so large that $E_t\le \delta \vro$.

Using \eqref{varb}, by Chebyshev inequality we have
\begin{eqnarray} \nonumber
   \P( |Y^\circ(t)|\ge \delta\vro) &\le& \frac{\Var Y^\circ(t)}{(\delta\vro)^2}
   \le \frac{ D_2 t (h\vro)^{2-\gamma}}{(\delta\vro)^2}
\\   \label{Cheb}
   &=&  \frac{ D_2 h^{2-\gamma}}{\delta^2} \ (t^{-1/\gamma}\vro)^{-\gamma} \to 0.
\end{eqnarray}
It is also useful to notice that for each $\rho>0$ and all large $t$ 
\begin{eqnarray*}
  \mulr_t[v_0,\infty) &=& \mulr_t[h\vro,\infty) = \mulr_t[h (t^{-1/\gamma}\vro) t^{1/\gamma},\infty)
\\
  &\le& \mulr_t[h \rho  t^{1/\gamma},\infty) 
  \le  \frac{\E(R^\gamma)}{\gamma(\gamma-1)}\,  (h \rho)^{-\gamma},
\end{eqnarray*}
where we used \eqref{mulr_b} at the last step. Letting $\rho\to\infty$ we get
$\mulr_t[v_0,\infty) \to 0$.

 Now we may proceed with the required lower bound as follows:
 \begin{eqnarray*}
  && \P( Y_{Q,\gamma}(t)\ge \vro) \ge \P(  |Y^\circ(t)|\le \delta\vro, Y^{\dag}(t)\ge (1+2\delta)\vro)
 \\
  &\ge& \P(  |Y^\circ(t)|\le \delta\vro) \ 
   \P( Y^{\dag}(t)\ge (1+2\delta)\vro; N(v_0,\infty)=1)
 \\
   &=& \P(  |Y^\circ(t)|\le \delta\vro) \,
       \exp\{-Q\mulr_t[v_0,\infty)\} \, Q \,\mulr_t[(1+2\delta)\vro,\infty).
 \end{eqnarray*}
 The idea behind this bound is to take a single service process providing a substantial large deviation
 workload and to suppress other contributions.

As we have just seen, the first two factors tend to one, thus
\be \label{LowerLD}
   \P( Y_{Q,\gamma}(t)\ge \vro) \ge Q \,\mulr_t[(1+2\delta)\vro,\infty) \ (1+o(1)).
\ee

\subsection{An upper bound for large deviations} 
Starting again with representation \eqref{split},  using $E_t\ge 0$ and \eqref{Cheb} we have
 \begin{eqnarray} \nonumber
     && \P( Y_{Q,\gamma}(t)\ge \vro) 
\\ \nonumber
     &\le& \P(  Y^\circ(t)\ge \delta\vro)   + \P(N[v_0,\infty)\ge 2)
            + \P( Y^{\dag}(t)\ge (1-\delta)\vro; N[v_0,\infty)=1)
 \\ \nonumber
     &=& \P(  Y^\circ(t) \ge \delta\vro) + \P(N[v_0,\infty)\ge 2)
           +  \P(N[(1-\delta)\vro,\infty)=1)
\\   \label{UpperLD}
      &\le& \frac{ D_2 t h^{2-\gamma}}{\delta^2\vro^\gamma} 
            + \frac 12 \left( Q\mulr_t[v_0,\infty)\right)^2 +Q\mulr_t[(1-\delta)\vro,\infty).
 \end{eqnarray}
Here the  last term is the main one. Recall that almost the same expression 
also shows up in the lower bound.

\subsection{Proof of Theorem \ref{t:moderate} \label{ss:proof_thm_moderate}}

Recall that, according to \eqref{mulr_asymp}, in the zone under consideration
$t^{1/\gamma}\ll \vro\ll t$, it is true that
\be \label{mulr_asymp_moder}
  \mulr_t[\vro,\infty) =   \frac{\E(R^\gamma)}{\gamma} \ t\ \vro^{-\gamma} (1+o(1))
\ee
and we have the similar representations with $\vro$ replaced by either $(1+2\delta)\vro$,
$(1-\delta)\vro$, or $v_0=h\vro$. 

In view of \eqref{mulr_asymp_moder}, the lower estimate \eqref{LowerLD} yields
\[
   \liminf_{t\to\infty} \frac{\P( Y_{Q,\gamma}(t)\ge \vro) }{t \vro^{-\gamma}}
   \ge  \frac{Q\ \E(R^\gamma)}{\gamma} (1+2\delta)^{-\gamma},
\]
while the upper estimate \eqref{UpperLD} yields
\[
   \limsup_{t\to\infty} \frac{\P( Y_{Q,\gamma}(t)\ge \vro) }{t \vro^{-\gamma}}
   \le \frac{D_2 h^{2-\gamma}}{\delta^2} + \frac{Q\ \E(R^\gamma)}{\gamma} (1-\delta)^{-\gamma},
\]
because the second term in \eqref{UpperLD} has a lower order of magnitude.

Letting first $h\to 0$, then $\delta\to 0$, we obtain
\[
   \lim_{t\to\infty} \frac{\P( Y_{Q,\gamma}(t)\ge \vro) }{t \vro^{-\gamma}}
   =  \frac{Q\ \E(R^\gamma)}{\gamma}, 
\]
as required.
\hfill $\Box$

\subsection{Proof of Theorem \ref{t:intermed} \label{ss:proof_thm_intermed}}

The proof goes along the same lines as in the moderate deviation case, except for the evaluation of $\mulr_t[\vro,\infty)$. Instead of \eqref{mulr_asymp_moder}, 
we have the following non-asymptotic exact formula. According to \eqref{mul},
for $\rho=\kappa t$ we have
\begin{eqnarray*}
    \mulr_t[\vro,\infty)  &=& \int_0^\infty \mul_t\left[  \frac{\vro}{r}, \infty \right) F_R(dr)
\\
     &=& \int_\kappa^\infty \left(
         t\ \frac{(\vro/r)^{-\gamma}}{\gamma} + \frac{2-\gamma}{(\gamma-1)\gamma}\, (\vro/r)^{1-\gamma}
   \right) F_R(dr)
\\
 &=&  \int_\kappa^\infty \left(
         \frac{\kappa^{-\gamma}}{\gamma} \ r^{\gamma} + \frac{(2-\gamma)\kappa^{1-\gamma} }{(\gamma-1)\gamma}\, r^{\gamma-1}
   \right) F_R(dr) \,  t^{-(\gamma-1)}
\\
 &=&   \left(
         \frac{\kappa^{-\gamma}}{\gamma}\, \E(R^{\gamma}\ed{R\ge\kappa}) 
         + \frac{(2-\gamma)\kappa^{1-\gamma} }{(\gamma-1)\gamma}\, \E(R^{\gamma-1}\ed{R\ge\kappa}) 
   \right) \,  t^{-(\gamma-1)} 
\\
 &=&   D^{(1)}_I(\kappa) \  t^{-(\gamma-1)}. 
\end{eqnarray*}
The latter constant is positive due to assumption $\P(R\ge \kappa)>0$.

For the lower bound, the estimate \eqref{LowerLD} yields
\[
    \P( Y_{Q,\gamma}(t)\ge \kappa t) 
    \ge Q \, D^{(1)}_I((1+2\delta)\kappa)\,  t^{-(\gamma-1)}    \ (1+o(1)).
\]
Letting $\delta\searrow 0$ and using \eqref{pkappa0}, we have
\begin{eqnarray*}
    \lim_{\delta\searrow 0} D^{(1)}_I((1+2\delta)\kappa)
     &=&  \frac{\kappa^{-\gamma}}{\gamma}\, \E(R^{\gamma}\ed{R >\kappa}) 
         + \frac{(2-\gamma)\kappa^{1-\gamma} }{(\gamma-1)\gamma}\, \E(R^{\gamma-1}\ed{R>\kappa})
\\
     &=&   D^{(1)}_I(\kappa).   
\end{eqnarray*}
Therefore,
\[
    \P( Y_{Q,\gamma}(t)\ge \kappa t) \ge Q \, D^{(1)}_I(\kappa)\,  t^{-(\gamma-1)} \ (1+o(1)),
\]
as required.

For the upper bound, the estimate \eqref{UpperLD} with $\vro=\kappa t$ yields
\[
    \P( Y_{Q,\gamma}(t)\ge \kappa t) 
    \le \left( \frac{D_2 h^{2-\gamma}}{\delta^2 \kappa^\gamma} + Q \, D^{(1)}_I((1-\delta)\kappa)\right)   t^{-(\gamma-1)}    \ (1+o(1)).
\]
Letting first $h\searrow 0$, we get rid of the first term and obtain
\[
    \P( Y_{Q,\gamma}(t)\ge \kappa t) 
    \le Q \, D^{(1)}_I((1-\delta)\kappa)\,  t^{-(\gamma-1)}    \ (1+o(1)).
\]
Letting $\delta\searrow 0$,  we have
\[
    \lim_{\delta\searrow 0} D^{(1)}_I((1-\delta)\kappa)=   D^{(1)}_I(\kappa).   
\]
Therefore,
\[
    \P( Y_{Q,\gamma}(t)\ge \kappa t) \le Q \, D^{(1)}_I(\kappa)\,  t^{-(\gamma-1)} \ (1+o(1)),
\]
as required.
\hfill $\Box$

\subsection{Proof of Theorem \ref{t:intermed2}} 
\label{ss:proof_thm_intermed2}

In the setting of this theorem, the large deviation probabilities decay faster with $t$ than Chebyshev inequality  \eqref{Cheb} suggests. Therefore, we need a finer estimate for $Y^\circ(t)$ given in the following lemma.

\begin{lem} \label {l:expest}
	For every $m,\delta > 0$ there exist $h>0$ and $C=C(h,\delta)>0$ such that for all $t>0$
\[
			\P(Y^\circ(t)\ge \delta t) \le C \, t^{-(\gamma-1)m},
\]
where $Y^\circ(t)$ is defined by $\eqref{split}$ with the splitting point $v_0:=ht$.
\end{lem}

\begin{proof}[ of the Lemma]
We start with some calculations valid for arbitrary $v_0$. We have the following formula for exponential moment of the centered Poisson integral:
\be \label{expmom}
	\E \exp (\lambda Y^\circ(t)) = \exp \big\{ \int_0^{v_0} (e^{\lambda v} - 1 - \lambda v) \mulr_t(dv) \big\}.
\ee
Let us split the integration domain in \eqref{expmom} into two  parts: $[0,v_0/2]$ and $(v_0/2,v_0]$.  
For the second one we have
\be \label{eq:int_zone2}
	\int_{\frac{v_0}{2}}^{v_0} (e^{\lambda v} - 1 - \lambda v) \mulr_t(dv) 
	\le e^{\lambda v_0} \cdot \mulr_t[\frac{v_0}{2},v_0]
	\le D_3 e^{\lambda v_0} t v_0^{-\gamma},
\ee
where $D_3 := \frac{2^\gamma \, \E (R^\gamma)}{\gamma(\gamma-1)}$. At the last step we used 
\eqref{mulr_b}.

For the first zone, by using inequality $e^x - 1 - x \le x^2 e^x$ and \eqref{mulr_b} we have
\begin{eqnarray*}
 &&	\int_0^{\frac{v_0}{2}} (e^{\lambda v} - 1 - \lambda v) \mulr_t(dv) 
	\le \int_0^{\frac{v_0}{2}} \lambda^2 v^2 e^{\lambda v} \mulr_t(dv)
\\
	&\le& 2 \, e^{\lambda v_0/2} \lambda^2 \int_0^{\frac{v_0}{2}} v \mulr_t[v,v_0/2] dv
	\le    \frac{2 \, \E (R^\gamma)}{\gamma(\gamma-1)}\,  e^{\lambda v_0/2} \lambda^2  t \int_0^{\frac{v_0}{2}} v^{1-\gamma} dv
\\	
	&=& D_4\, e^{\lambda v_0/2} \, \lambda^2\, t\, v_0^{2-\gamma},
\end{eqnarray*}
where $D_4 := \tfrac{2^{\gamma-1} \, \E (R^\gamma)}{\gamma(\gamma-1)(2-\gamma)}$. 

Next, using inequality $e^{\frac{x}{2}}x^2 \le 3 e^x$, we have
\be \label{eq:int_zone1}
	\int_0^{\frac{v_0}{2}} (e^{\lambda v} - 1 - \lambda v) \mulr_t(dv) 
	\le 3\, D_4\, e^{\lambda v_0}\, t\, v_0^{-\gamma}.
\ee
By summing up \eqref{eq:int_zone2} and  \eqref{eq:int_zone1}, we obtain
\[
   	\E \exp (\lambda Y^\circ(t)) \le \exp\left\{ 
   	    \left( D_3+ 3\, D_4\right) t\, v_0^{-\gamma} e^{\lambda v_0} \right\}
   	:= \exp\left\{  A e^{\lambda v_0} \right\},
\]
where
\[
   A := (D_3+3D_4)tv_0^{-\gamma}. 
\]
For every real $y$ by exponential Chebyshev inequality we have
\be  \label{expCheb}
	\P(Y^\circ(t)\ge y) 
	\le \inf\limits_{\lambda > 0} \ 
	 \exp(A e^{\lambda v_0} -  \lambda y).
\ee

If $y>Av_0$, the minimum on the right hand side is attained at the point 
$\lambda = \frac{1}{v_0} \log(\frac{y}{Av_0})$. 
By plugging this value in \eqref{expCheb} we obtain
\begin{eqnarray} \nonumber 
	\P(Y^\circ(t)\ge y) &\le& 
    \exp\left(\frac{y}{v_0}\right) \ \left(\frac{Av_0}{y}\right)^{\frac{y}{v_0}} 
\\ \label{expCheb2}
  &=& \exp\left(\frac{y}{v_0}\right) \ \left(   (D_3+3D_4)\,\frac{t v_0^{1-\gamma}}{y}\right)^{\frac{y}{v_0}}. 
\end{eqnarray}
Letting here $y:=\delta t$, $v_0:=ht$ yields
\[
   	\P(Y^\circ(t)\ge \delta t)  \le
   		 C \, t^{-\frac{\delta}{h}(\gamma-1)},
\]
where $C$ depends only on $\delta, h$.
Choosing $h < \frac{\delta}{m}$ we get the result.
\end{proof}
\medskip

Now we can proceed to the proof of the theorem.
\medskip

{\bf Upper bound.} Let $\eta:=\tfrac{(1-\zeta)\kappa}{n-\zeta}$.
Since $\zeta\in (0,1)$, we have $\eta>0$. It also follows from the
definition that $\tfrac{\kappa}{n-\zeta}=\tfrac{\kappa-\eta}{n-1}$.
Therefore, we may rewrite \eqref{choosezeta} as
\be \label{chooseeta}
  \P\left(R\ge \frac{\kappa-\eta}{n-1}\right) =0.
\ee
Let $\delta\in(0,\eta)$. By using Lemma \ref{l:expest} with $m=n+1$ we find a small $h>0$ such that 
\be \label{chooseh}
  \P(Y^\circ(t)\ge \delta t) \le C\, t^{-(\gamma-1)(n+1)}.
\ee
By using the decomposition \eqref{split} with  $v_0=ht$ and taking into account $E_t\ge 0$ we get the bound
\[
   \P\left(Y_{Q,\gamma}(t)\ge \kappa t\right)
   \le 
    \P\left(Y^\circ(t)\ge \delta t\right)
    + \P\left(Y^{\dag}(t)\ge (\kappa-\delta) t\right).
\]
The first term is negligible by \eqref{chooseh}. 
Let denote $N_0:=N[v_0,\infty)$, 
which is a Poissonian random variable with intensity
$\mu_0:=Q \mulr_t[v_0,\infty)$,
and apply the following bound to the second term:
\begin{eqnarray*} 
  && \P\left(Y^{\dag}(t)\ge (\kappa-\delta) t\right)
  \le \P(N_0>n)
\\
  &&  + \
      \P\left(Y^{\dag}(t)\ge (\kappa-\delta) t\right); N_0=n)+
      \P\left(Y^{\dag}(t)\ge (\kappa-\delta) t\right), N_0\le n-1).
\end{eqnarray*}

For the first term, an elementary bound for Poisson tail
works, namely
\[
 \P(N_0>n)=e^{-\mu_0}\sum_{m=0}^\infty \frac{\mu_0^{n+1+m}}{(n+1+m)!}
 \le e^{-\mu_0} \frac{\mu_0^{n+1}}{(n+1)!}
 \sum_{m=0}^\infty \frac{\mu_0^{m}}{m!}
 \le \frac{\mu_0^{n+1}}{(n+1)!}
\]
where we used that $(n+1+m)!\ge (n+1)! m!$.
Notice that by \eqref{mulr_b} with $\rho:=v_0=ht$ we have
\[
   \mu_0 \le  \frac{Q \,\E(R^\gamma)} {\gamma(\gamma-1)} \, t\, (ht)^{-\gamma}
   = \frac{Q\, \E(R^\gamma)h^{-\gamma} } {\gamma(\gamma-1)} \, t^{-(\gamma-1)},
\]
hence,
\[
  \P(N_0>n) = O\left( t^{-(\gamma-1)(n+1)}\right)
\]
is negligible.

Further, by using \eqref{chooseeta} and the definition of the measure $\mulr_t$ we see that
\[
   \mulr_t[\tfrac{(\kappa-\delta)t}{n-1},\infty)
   \le  \mulr_t[\tfrac{(\kappa-\eta)t}{n-1},\infty)
   =0,
\]   
which implies
\[ 
    \P\left(Y^{\dag}(t)\ge (\kappa-\delta) t\right),   N_0\le n-1)=0,
\]
because here Poissonian integral $Y^{\dag}(t)$ is a sum of not more than $n-1$ terms each being strictly smaller than $\tfrac{(\kappa-\delta)t}{n-1}$.

For $A \in \BB ([v_0, \infty))$ denote $N_A:=N(A)$ with intensity $\mu_A := Q \mulr_t(A)$ 
	and $\nu_t^{(l,r)}(A) := \P(N_A = 1 \ | \ N_0 = 1)$, which is a measure on $[v_0, \infty)$. We have 
	\[
	\nu_t^{(l,r)}(A) = e^{-\mu_A} \mu_A \cdot e^{\mu_A - \mu_0} \cdot \frac {e^{\mu_0}} {\mu_0} = \frac{\mu_A}{\mu_0}. 
	\]
	\\
	The remaining Poissonian integral with fixed number of points admits the following representation
	\begin{eqnarray*}
		&& \P\left(Y^{\dag}(t)\ge (\kappa-\delta) t\right); N_0=n)
		\\
		&=& \P\left(Y^{\dag}(t)\ge (\kappa-\delta) t\right) \ | \  N_0=n)
		\ \P(N_0) = n
		\\
		&=& e^{-\mu_0} \ \frac{\mu_0^n}{n!} \int_{[v_0,\infty)^n} \ed{v_1+...+v_n\ge (\kappa-\delta)t} \prod_{m=1}^n \nu_t^{(l,r)}(dv_m)
		\\
		&\le& e^{-\mu_0} \ \frac{Q^n}{n!} \int_{\R_+^n} \ed{v_1+...+v_n\ge (\kappa-\delta)t} \prod_{m=1}^n \mulr_t(dv_m)
		\\   
		&=& e^{-\mu_0} \ \frac{Q^n}{n!} \int_{[0,t]^n}\int_{R_+^n} \ed{\ell_1 r_1+...+\ell_n r_n\ge (\kappa-\delta)t} \prod_{m=1}^n \mul_t(d\ell_m)
		\prod_{m=1}^n F_R(dr_m)
		\\   
		&=& e^{-\mu_0} \ \frac{Q^n}{n!} \int_{[0,1]^n}
		\P\left(s_1 R_1+...+s_n R_n\ge \kappa-\delta\right) \prod_{m=1}^n \nu(ds_m)\ t^{-(\gamma-1)n},
		\\
		&:=& e^{-\mu_0} \ Q^n D_I^{(n)}(\kappa-\delta) \ t^{-(\gamma-1)n},
	\end{eqnarray*}
where $R_1, ..., R_n$ are i.i.d. variables with distribution $F_R$ and, according to \eqref{mul}, $\nu$ is a measure on $[0,1]$ having the weight 
$\tfrac{1}{\gamma(\gamma-1)}$ at $1$ and the density
\[
    \frac{d\nu}{ds}\, (s)= s^{-(\gamma+1)}+\frac{2-\gamma}{\gamma}\, s^{-\gamma}, \qquad 0<s<1.
\]

Notice also that the constant  $D^{(n)}_I(\kappa-\delta)$ is finite although measure $\nu$ is infinite at each neighborhood of zero. The reason is that the probability we integrate vanishes if for some $m$ one has
$s_m<s_*:=\tfrac{(n-1)(\eta-\delta)}{\kappa-\eta}$ where $\eta>0$ satisfies \eqref{chooseeta}. Indeed, we have in this case
\begin{eqnarray*}
   \P\left(s_1 R_1+...+s_n R_n\ge \kappa-\delta \right)
   &\le&  \P\left((s_*+(n-1)) \max_{1\le m\le n} R_j\ge \kappa-\delta \right)
\\
    &\le& n \, \P\left( R\ge \frac{\kappa-\delta}{s_*+(n-1)}\right)
\\     &=& n\, \P\left( R\ge \frac{\kappa-\eta}{n-1}\right)
     = 0.
\end{eqnarray*}

We summarize our findings as
\[
    \P\left(Y_{Q,\gamma}(t)\ge \kappa t\right)
    \le Q^n D_I^{(n)}(\kappa-\delta) \ t^{-(\gamma-1)n}
    (1+o(1)).
\]

Letting $\delta\searrow 0$, we obtain
\[
  \P\left(Y_{Q,\gamma}(t)\ge \kappa t\right)
 \le Q^n\, D^{(n)}_I(\kappa) \ t^{-(\gamma-1)n}
 (1+o(1)),
\]
 where
 \begin{eqnarray} \nonumber
    D^{(n)}_I(\kappa)
    &:=& \lim_{\delta\to 0} D^{(n)}_I(\kappa-\delta)
\\ \label{DIn}
  &=&\frac{1}{n!} \int_{[0,1]^n}
  \P\left(s_1 R_1+...+s_n R_n\ge \kappa\right) \prod_{m=1}^n \nu(ds_m).
 \end{eqnarray}
 
It is easy to see that for $n=1$ we obtain the same value of $ D^{(1)}_I(\kappa)$ as in Theorem \ref{t:intermed}.
\bigskip

{\bf Lower bound.} 

First, notice that $E_t$ in \eqref{split} is still negligible because by \eqref{Ebound} we have
\[ 
  E_t \le D_1\, t \, v_0^{1-\gamma} =D_1 \, t\, (ht)^{1-\gamma}
  = O(t^{2-\gamma}) =o(t).
\] 
Hence, for every fixed small $\delta$ we may and do assume that $E_t\le \delta t$ for large $t$.

Second, using \eqref{varb}, by Chebyshev inequality we have
\begin{eqnarray} \nonumber
   \P( |Y^\circ(t)|\ge \delta\, t) &\le& \frac{\Var Y^\circ(t)}{(\delta\, t)^2}
   \le \frac{ D_2 t (h t)^{2-\gamma}}{(\delta\, t)^2}
\\   \nonumber
   &=&  \frac{ D_2 h^{2-\gamma}}{\delta^2} \ t^{-(\gamma-1)} \to 0, \quad\textrm{as } t\to\infty.
\end{eqnarray}
Therefore, we may proceed towards the required lower bound as follows:
 \begin{eqnarray*}
  && \P( Y_{Q,\gamma}(t)\ge \kappa\,t) \ge 
  \P(  |Y^\circ(t)|\le \delta\,t, Y^{\dag}(t)\ge (\kappa+2\delta)t)
 \\
  &\ge& \P(  |Y^\circ(t)|\le \delta\, t) \ 
   \P( Y^{\dag}(t)\ge (\kappa+2\delta) t; N_0=n)
 \\
   &=&  (1+o(1)) \ 
   \P( Y^{\dag}(t)\ge (\kappa+2\delta) t; N_0=n).
 \end{eqnarray*}
 The idea behind this bound is to focus on $n$ service process providing a substantial large deviation
 workload and to suppress other contributions.
 
 Furthermore, by using the expression obtained while working
 on the upper bound
\begin{eqnarray*}
   && \P\left(Y^{\dag}(t)\ge (\kappa+2\delta) t\right); N_0=n)
\\   
   &=& \frac{Q^n}{n!} \int_{[0,1]^n}
  \P\left(s_1 R_1+...+s_n R_n\ge \kappa+2\delta\right) \prod_{m=1}^n \nu(ds_m)\ t^{-(\gamma-1)n} (1+o(1))
\\
  &:=& Q^n \, D_I^{(n)}(\kappa+2\delta)t^{-(\gamma-1)n} (1+o(1)).
\end{eqnarray*}
By letting $\delta\searrow 0$, we obtain
\begin{eqnarray*}
   \lim_{\delta\searrow 0} D_I^{(n)}(\kappa+2\delta)
   &=&  \frac{1}{n!} \int_{[0,1]^n}
  \P\left(s_1 R_1+...+s_n R_n> \kappa\right) \prod_{m=1}^n \nu(ds_m)
\\
   &=&  \frac{1}{n!} \int_{[0,1]^n}
  \P\left(s_1 R_1+...+s_n R_n\ge \kappa\right) \prod_{m=1}^n \nu(ds_m)
  \\ 
    &=&  D_I^{(n)}(\kappa).
\end{eqnarray*}
For the  non-obvious passage we have used the following lemma.

\begin{lem} \label{l:nunzero}
    Assume that \eqref{pkappa0n} holds. Then
\be \label{nunzero}
   \nu^n\left\{ \ss:= (s_1,...,s_n): \P(s_1R_1+...+s_nR_n=\kappa)>0\right\} =0.
\ee    
\end{lem}

The required lower bound
\[ 
    \P( Y_{Q,\gamma}(t)\ge \kappa\,t) \ge
    Q^n  D^{(n)}_I(\kappa)\,  t^{-(\gamma-1)n} (1+o(1))
\]
follows now from the previous estimates.
It  merely remains to prove the lemma.
\hfill $\Box$
\medskip

\begin{proof}[ of Lemma \ref{l:nunzero} ]
Let $r_1,...,r_n$ be the atoms of the distribution $F_R$, i.e. $\P(r_m)>0, 1\le m\le n$.
Define
\[
   F=F(r_1,...,r_n):=\{ \ss\in [0,1]^n: s_1r_1+...+s_nr_n=\kappa \}.
\]

For every subset of integers $M\subset[1..n]$ let
\[
   B_M:=\{ \ss\in [0,1]^n: s_m\in [0,1), m\in M; s_m=1, m\not \in M \}.
\]
Notice that $[0,1]^n=\cup_{M} B_M$.
Let 
\[
  F_M:=F\bigcap B_m= \{\ss\in B_M: \sum_{m\in M} s_m r_m = \kappa- \sum_{m\not\in M} r_m\}.
\]
If $M$ is not empty, then $\nu^n(F_M)=0$ because $\nu$ is absolutely continuous on $[0,1)$.

If $M$ is empty, then $B_M=\{(1,...,1)\}$ is a singleton and $F_M=\emptyset$ because otherwise $\sum_{m=1}^n r_m=\kappa$ which would contradict to \eqref{pkappa0n}.

We conclude that 
\[
   \nu^n\left( F(r_1,...,r_n)\right)= \sum_M \nu^n(F_M) =0.
\]

Since 
\[
  \left\{ \ss: \P(s_1R_1+...+s_nR_n=\kappa)>0\right\} 
  \subset \bigcup_{r_1,...,r_n}  F(r_1,...,r_n)
\]
and the union is countable, we obtain \eqref{nunzero}.

\end{proof}

\subsection{Proof of Theorem \ref{t:ultra_regular}}

{\bf Upper bound.}

 We take a small $\delta>0$, use decomposition \eqref{split} with $v_0:= h \vro$ (a small $h=h(\delta)$ will be specified later on), and start with a usual bound
\begin{eqnarray*}
  && \P\left(Y_{Q,\gamma}(t)\ge \vro\right)
   \le 
    \P\left(Y^\circ(t)\ge \delta \vro \right)
    + \P\left(Y^{\dag}(t)\ge (1-\delta) \vro \right)
    \\
     &\le& 
     \P\left(Y^\circ(t)\ge \delta \vro \right)
    + \P\left(Y^{\dag}(t)\ge (1-\delta) \vro ; N_0=1\right)
    +P(N_0\ge 2).
\end{eqnarray*}
To show that the first term is negligible we use estimate
\eqref{expCheb2} with $y:=\delta\vro$, $v_0:=h\vro$ and obtain for some $C=C(\delta,h)$
\[
     \P\left(Y^\circ(t)\ge \delta \vro \right)
     \le C \left(t\vro^{-\gamma}\right)^{\frac{\delta}{h}}
     \le  C \vro^{-(\gamma-1)\,\frac{\delta}{h}}
     \ll \tpf(\vro)\le \tpf(\vro/t)
\]
whenever $h$ is chosen so small that 
$(\gamma-1)\,\frac{\delta}{h}>m$.
 
Subsequent evaluation of $Y^\dag(t)$ requires analysis of the measure $\mulr_t$. By using \eqref{mulr_intrepr} and \eqref{mul} we obtain
\begin{eqnarray*}
 \mulr_t[v,\infty) &=& \int_{v/t}^\infty \mul_t\left[\frac vr,t\right] F_R(dr)
 \\
  &=& \int_{v/t}^\infty 
  \left(  \frac{t\,(r/v)^{\gamma}}{\gamma} +\frac{2-\gamma}{(\gamma-1)\gamma} \, (r/v)^{\gamma-1}
  \right)F_R(dr)
  \\
   &=&  \frac{t\,v^{-\gamma}}{\gamma} 
   \int_{v/t}^\infty r^\gamma F_R(dr)
   +\frac{(2-\gamma)v^{1-\gamma}}{(\gamma-1)\gamma} 
   \int_{v/t}^\infty r^{\gamma-1}  F_R(dr).
\end{eqnarray*}
Since the tail of $F_R$ is regularly varying, we have the following asymptotics for the integrals
\begin{eqnarray*}
   \int_{z}^\infty r^\gamma F_R(dr) &=& \frac{m z^\gamma}{m-\gamma}\,\tpf(z)\, (1+o(1)),
\\
  \int_{z}^\infty r^{\gamma-1} F_R(dr) &=& \frac{m z^{\gamma-1}}{m-\gamma+1}\,\tpf(z) \, (1+o(1)),
  \quad \textrm{as } z\to \infty.
\end{eqnarray*}
Therefore, we obtain
\begin{eqnarray} \nonumber
   \mulr_t[v,\infty) &=& t^{-(\gamma-1) }
 \left [
      \frac{m}{\gamma(m-\gamma)}  
   +
   \frac{(2-\gamma)m}{(\gamma-1)\gamma(m-\gamma+1)} 
   \right]
 \tpf(v/t) \, (1+o(1))
\\  \nonumber
  &=&  
  \frac{m(m-1)}{\gamma(\gamma-1)(m-\gamma+1)(m-\gamma)}
  \, t^{-(\gamma-1)} \tpf(v/t) \, (1+o(1)).
\\   \label{mulr_asymp_ultra}
  &=& 
  D \, t^{-(\gamma-1)} \tpf(v/t) \, (1+o(1)),
  \qquad \textrm{as }  v\gg t.
\end{eqnarray}
Now the evaluation of $Y^\dag$ is straightforward. Indeed,
by \eqref{mulr_asymp_ultra}
\begin{eqnarray*}
 &&  \P\left(Y^{\dag}(t)\ge (1-\delta) \vro ; N_0=1\right)
   \le Q \, \mulr_t[(1-\delta) \vro, \infty)
\\
   &=&  Q\, D \, t^{\gamma-1} \tpf((1-\delta)\vro/t) \, (1+o(1))
\\
   &=&  Q\, D \, t^{\gamma-1} \tpf(\vro/t) (1-\delta)^{-m}\, (1+o(1))
\end{eqnarray*}
and
\begin{eqnarray*}
    &&  P(N_0\ge 2)  \le Q^2 \, \mulr_t[h\vro,\infty)^2
\\
    &=& Q^2\, (D t^{-(\gamma-1)} \tpf(h\vro/t))^2 (1+o(1))
\\
    &=&  Q^2 \, (D t^{-(\gamma-1)} \tpf(\vro/t) h^{-m})^2 (1+o(1))
   \ll t^{-(\gamma-1)} \tpf(\vro/t). 
\end{eqnarray*}

By combining these estimates and letting $\delta\to 0$ we obtain the desired bound
\[
    \P\left(Y_{Q,\gamma}(t)\ge \vro\right)
  \le Q\,  D \, t^{-(\gamma-1)} \tpf(\vro/t) \, (1+o(1)).
\]
\medskip

 {\bf Lower bound.} 
 
 Since  $\vro\gg t \gg t^{1/\gamma}$, all bounds from section
 \ref{ss:LowerLD} apply. For every  $\delta>0$ inequality \eqref{LowerLD} along with \eqref{mulr_asymp_ultra} yield
 \begin{eqnarray*}
   \P\left(Y_{Q,\gamma}(t)\ge \vro\right)
   &\ge&   Q \, \mulr_t[(1+2\delta)\vro,\infty) (1+ o(1))
\\  
  &=& Q \, D \, t^{-(\gamma-1)} (1+2\delta)^{-m} \tpf(\vro/t) \, (1+o(1)),
 \end{eqnarray*}
 and letting $\delta\to 0$ we get the desired bound
 \[
    \P\left(Y_{Q,\gamma}(t)\ge \vro\right)
    \ge Q\, D \, t^{-(\gamma-1)} \tpf(\vro/t) \, (1+o(1)).
\]
 \hfill $\Box$


\section*{Acknowledgement.} 
This work was supported by Russian Science Foundation grant 21-11-00047.



\begin{thebibliography}{99}

{\baselineskip=10pt \small

\bibitem{CT}
Cohen, S. and Taqqu, M. (2004). Small and large scale behavior
of the Poissonized Telecom Process,
Methodol. Comput. Appl. Probab. \textbf{6}, pp. 363--379.

\bibitem{DZ} Dembo, A., Zeitouni, O.
 Large Deviations Techniques and Applications, Springer, 2010.
 
\bibitem{Gai}
Gaigalas, R. (2006). A Poisson bridge between fractional Brownian
motion and stable L\'evy motion, Stoch. Proc. Appl.
\textbf{116}, pp. 447--462.

\bibitem{K02}
Kaj, I. (2002). Stochastic Modeling in Broadband Communications Systems,
SIAM Monographs on Mathematical Modeling and Computation Vol.8 (SIAM,
Philadelphia).

\bibitem{K05}
Kaj, I. (2005). Limiting fractal random processes in heavy-tailed systems,
Fractals in Engineering, New Trends in Theory and Applications,
J. Levy-Vehel, E. Lutton (eds.), pp. 199--218 (Springer-Verlag, London).

\bibitem{K06}
Kaj, I. (2006). Aspects of Wireless Network Modeling Based
on Poisson Point Processes, Fields Institute Workshop on
Applied Probability (Carleton University, Ottawa).


\bibitem{KLNS}
Kaj, I., Leskel\"a, L., Norros, I., and Schmidt, V. (2007).
Scaling limits for random fields with long-range dependence,
Ann. Probab. \textbf{35}, pp. 528--550.

\bibitem{KT}
Kaj, I. and Taqqu, M.~S. (2008). Convergence to fractional Brownian
motion and to the Telecom process: the integral representation approach,
In and Out of Equilibrium. II., ser.: Progress in Probability,
Vol.~60,(Birkh\"auser, Basel), pp. 383--427.

\bibitem{Kur}
Kurtz, T.~G. (1996). Limit theorems for workload
input models,  Stochastic Networks, Theory and Applications
(Clarendon Press, Oxford) Kelly, F.~P., Zachary, S. and Ziedins, I.
(eds.), pp. 119--140.

\bibitem{LifEx} Lifshits, M. (2014) Random Processes by Example. World Scientific, Singapore. 

\bibitem{PipT0}
Pipiras V. and Taqqu, M.~S. (2000). The limit of a renewal-reward
process with heavy-tailed rewards is not a linear fractional stable motion,
Bernoulli \textbf{6}, pp. 607--614.

\bibitem{RH}
Rosenkrantz, W.~A. and  Horowitz, J. (2002). The infinite sourse model
for internet traffic: statistical analysis and limit theorems,
Methods and Applications of Analysis
\textbf{9}, pp. 445--462.

\bibitem{Taq02} Taqqu, M.~S. (2002). The modeling of Ethernet data and of
signals that are heavy-tailed with infinite variance, Scand. J. Statist.
\textbf{29}, pp. 273--295.

} 

\end{thebibliography}
\end{document}